\documentclass{aptpub}
\usepackage{eucal,mathrsfs,dsfont}
\authornames{JIAN WANG} 
\shorttitle{L\'{e}vy driven Ornstein-Uhlenbeck process} 

\newcommand{\R}{\mathds R}

\newcommand{\Rd}{{\mathds R^d}}
\newcommand{\Pp}{\mathds P}
\newcommand{\Ee}{\mathds E}
\newcommand{\I}{\mathds 1}

\newcommand{\Bb}{\mathscr{B}}
\newcommand{\scalp}[2]{\langle#1,#2\rangle}

\newcommand{\mmu}{\mathsf{\pi}}
\begin{document}

\title{On the exponential ergodicity for
L\'{e}vy driven Ornstein-Uhlenbeck processes} 

\authorone[Fujian Normal
University]{Jian Wang} 

\addressone{School of Mathematics and Computer Science, Fujian Normal
University, 350007, Fuzhou, P.R. China.} 

\emailone{jianwang@fjnu.edu.cn}

\begin{abstract}
Based on the explicit coupling property, the ergodicity and the exponential
ergodicity of L\'{e}vy driven Ornstein-Uhlenbeck processes are established.
\end{abstract}

\keywords{L\'{e}vy processes; Ornstein-Uhlenbeck
processes; coupling property; exponential
ergodicity} 

\ams{60H10}{60J75; 60G51} 

\section{Introduction and main results}\label{section1} 
Let $(X^x_t)_{t\ge0}$ be a $d$-dimensional Ornstein-Uhlenbeck
process, which is defined as the unique strong solution of the
following stochastic differential equation
\begin{equation}\label{ou1}
    dX_t = AX_t\,dt + dZ_t,\qquad X_0=x\in \R^d.
\end{equation}
Here $A$ is a real $d\times d$ matrix, and $(Z_t)_{t\ge0}$ is a
L\'{e}vy process in $\R^d$. It is well known that
$(X^x_t)_{t\ge0}$ is a strong Markov process with the following
form
\begin{equation}\label{ou1ou}
    X_t^x
    =e^{tA}x + \int_0^t e^{(t-s)A}\,dZ_s.
\end{equation}
The associated Markov semigroup acting on ${B}_b(\R^d)$, the class
of all bounded measurable functions on $\R^d$, is given by
\begin{equation} \label{eq4} P_tf(x):=\Ee f(X_t^x)=\int_{\R^d} f(e^{tA}x+z)\,\mmu_t(dz), \quad t\ge0, x\in \R^d, f\in {B}_b(\R^d),\end{equation}
where $\mmu_t$ is the law of $\int_0^t e^{(t-s)A} \,d Z_s$.
Semigroups of the type \eqref{eq4} are generalized Mehler
semigroups.

Let us recall that a L\'{e}vy process $Z=(Z_t)_{t\ge0}$ with
values in $\R^d$ is a $\R^d$-valued process defined on some
stochastic basis $(\Omega,\mathscr{F}, (\mathscr{F}_t)_{t\ge0}, \Pp)$,
continuous in probability, having stationary independent
increments, c\`{a}dl\`{a}g trajectories, and such that $Z_0=0$,
$\Pp$-a.s. It is well known that the characteristic exponent or
the symbol $\Phi$ of $(Z_t)_{t\ge0}$, defined by
$$
    \Ee\bigl(e^{i\scalp{\xi}{Z_t}}\bigr)
    =e^{-t\Phi(\xi)},\quad \xi\in\R^d,
$$
enjoys the following L\'{e}vy-Khintchine representation:
\begin{equation}\label{ou2}
     \Phi(\xi)
     =\frac{1}{2}\scalp{Q\xi}{\xi} +i\scalp{b}{\xi} +\int_{z\neq 0} \Bigl(1-e^{i\scalp{\xi}{z}}+i\scalp{\xi}{z}\I_{B(0,1)}(z)\Bigr)\nu(dz),
\end{equation}
where $Q\in \R^{d\times d}$ is a positive semi-definite matrix,
$b\in\Rd$ is the drift vector and $\nu$ is the L\'evy measure,
i.e.\ a $\sigma$-finite measure on $\R^d\setminus\{0\}$ such that
$\int_{z\neq 0}(1\wedge |z|^2)\,\nu(dz)<\infty$. Our main reference for L\'{e}vy processes is the monograph \cite{Sa}.

\medskip

The starting point of our paper is the following result about the existence of invariant measure for Ornstein-Uhlenbeck
processes, which was proven in \cite[Theorem 4.1]{SA1}.

\medskip

 \paragraph{{\bf Theorem 1.0}} \emph{Let $X=(X^x_t)_{t\ge0}$ be a $d$-dimensional Ornstein-Uhlenbeck
process determined by \eqref{ou1}, where the real parts of all eigenvalues of $A$ are
negative. If the L\'{e}vy measure $\nu$ of
L\'{e}vy process $Z$ satisfies $\int_{\{|z|\ge1\}}\log
(1+|z|)\,\nu(dz)<\infty$, then there exists an invariant measure $\mu$ such that for any $A\in \mathscr{B}(\R^d)$,
$$P_t(x,A)\to \mu(A),\quad t\to \infty,$$
where $P_t(x,dz)$ is the transition kernel of the process $(X_t)_{t\ge0}$. }

\medskip

We are very much interested in the ergodicity and the exponential
ergodicity of Ornstein-Uhlenbeck processes. The standard method
yielding the ergodicity is to verify that the process is strong Feller and
irreducible, cf.\ see \cite{Ci,MT, T}. The strong Feller property
of Ornstein-Uhlenbeck processes has been studied in \cite{BK,
Mu,PZ,S}. In particular, according to \cite[Theoerm 1.1 and Proposition
2.1]{PZ}, if the L\'{e}vy measure $\nu$ of L\'{e}vy process $Z$ is
infinite and has a density with respect to the Lebesgue measure,
then the associated Ornstein-Uhlenbeck process $(X_t)_{t\ge0}$
determined by \eqref{ou1} is strong Feller. We refer the readers to
see \cite[Section 3]{PZ1} for some discussions about the
irreducibility of Ornstein-Uhlenbeck processes. The novelty of this
paper is the direct use of coupling property in the proof of the
ergodicity (and the exponential ergodicity) for Ornstein-Uhlenbeck
processes. The coupling property of Ornstein-Uhlenbeck processes has
been studied in \cite{SW2,W1}. As we will see in the last section,
an obvious advantage of the coupling method lies in the succinctness
of the proof, which yields both the ergodicity and the exponential
ergodicity simply via the L\'{e}vy measure $\nu$.

\medskip

Before stating our main results, we first introduce some necessary
notations. Let $\nu$ be the L\'{e}vy measure of the L\'{e}vy
process $(Z_t)_{t\ge0}$, see \eqref{ou2}. For every
$\varepsilon>0$, define ${\nu}_\varepsilon$ on $\R^d$ as follows: for any $B\in \mathscr{B}(\R^d),$
\begin{equation}\label{levycut}
    {\nu}_\varepsilon(B)
    =
    \begin{cases}
        \nu(B),                           & \text{if\ \ } \nu(\R^d)<\infty;\\
        \nu(B\setminus \{z: |z|<\varepsilon\}), & \text{if\ \ } \nu(\R^d)=\infty.
    \end{cases}
\end{equation}Recall that for any two bounded measures $\mu_1$ and
$\mu_2$ on $(\R^d,\Bb(\R^d))$,
$$\mu_1\wedge\mu_2:=\mu_1-(\mu_1-\mu_2)^+,$$ where
$(\mu_1-\mu_2)^{\pm}$ refers to the Jordan-Hahn decomposition of the
signed measure $\mu_1-\mu_2$. In particular,
$\mu_1\wedge\mu_2=\mu_2\wedge\mu_1$, and $$\mu_1\wedge
\mu_2\,(\R^d)=\frac{1}{2}\bigg[\mu_1(\R^d)+\mu_2(\R^d)-\|\mu_1-\mu_2\|_{\var}\bigg],$$
where $\|\cdot\|_{\var}$ stands for the total variation norm. Denote
by $P_t(x,\cdot)$ the transition kernel of the process $X$.

\begin{theorem} \label{erg1}
Let $X=(X^x_t)_{t\ge0}$ be a $d$-dimensional Ornstein-Uhlenbeck
process determined by \eqref{ou1}, where the real parts of all eigenvalues of $A$ are
negative. Then we have the following two statements:

\begin{itemize}
\item [(1)] If the L\'{e}vy measure $\nu$ of
L\'{e}vy process $Z$ satisfies $\int_{\{|z|\ge1\}}\log
(1+|z|)\,\nu(dz)<\infty$ and \begin{equation}\label{erg11}\inf_{
|x|\le \rho}\nu _\varepsilon\wedge
(\delta_{x}*\nu_\varepsilon)(\R^d)>0\end{equation} for some
constants $\varepsilon, \rho>0$, then the process $X$ is
ergodic, i.e.\ there is a unique invariant measure $\mu$ such that
for any $x\in\R^d$,
$$\lim_{t\to \infty}\|P_t(x, \cdot)-\mu\|_{\var}=0.$$

\item[(2)] If the L\'{e}vy measure $\nu$ of
L\'{e}vy process $Z$ satisfies $\int_{\{|z|\ge1\}}|z|\,\nu(dz)<\infty$ and
\begin{equation}\label{erg12}\limsup_{\rho\to0}
\left[\frac{\sup\limits_{ |x|\le \rho}\|\nu _\varepsilon-
(\delta_{x}*\nu_\varepsilon)\|_{_{\var}}}{\rho}\right]<\infty\end{equation}
for some constant $\varepsilon>0$, then the process $X$ is
exponentially ergodic. More explicitly, there are a unique invariant
measure $\mu$ and two constants $\kappa, C>0$ such that for any
$x\in\R^d$ and $t>0$,
$$\|P_t(x, \cdot)-\mu\|_{\var}\le C (1+|x|)\exp\left(-\kappa\,t\right).$$
\end{itemize}
\end{theorem}

\begin{remark} (1) Under \eqref{erg12} and for fixed $\varepsilon>0$, there exists $\rho>0$ such that
$$\sup\limits_{ |x|\le \rho}\|\nu _\varepsilon-
(\delta_{x}*\nu_\varepsilon)\|_{_{\var}}\le \nu_\varepsilon(\R^d),$$ and so
$$\aligned\inf_{
|x|\le \rho}\nu _\varepsilon\wedge
(\delta_{x}*\nu_\varepsilon)(\R^d)&=\frac{1}{2}\inf_{
|x|\le \rho}\bigg[\nu_\varepsilon(\R^d)+(\delta_{x}*\nu_\varepsilon)(\R^d)-\|\nu _\varepsilon-
(\delta_{x}*\nu_\varepsilon)\|_{_{\var}}\bigg]\\
&=\frac{1}{2}\bigg[2\nu_\varepsilon(\R^d)-\sup_{
|x|\le \rho}\|\nu _\varepsilon-
(\delta_{x}*\nu_\varepsilon)\|_{_{\var}}\bigg]\\
&\ge \frac{1}{2}\nu_\varepsilon(\R^d)>0.\endaligned$$ This shows
that \eqref{erg12} implies \eqref{erg11}.

(2) We mention that in many applications the condition \eqref{erg11}
is weak. For example, it is proven in \cite[Proposition 1.5]{SW2}
that \eqref{erg11} is satisfied, when L\'{e}vy measure $\nu$ of
$(Z_t)_{t\ge0}$ satisfies $\nu(dz)\ge \rho(z)\,dz$ such that
$\int_{\{|z-z_0|\le \varepsilon\}}    \frac{dz}{\rho(z)}<\infty$
holds for some $z_0\in\R^d$ and some $\varepsilon>0$.

(3) According to Theorem \ref{thehreheh} below, the second assertion in Theorem \ref{erg1} still holds,
if \eqref{erg12} is replaced by
$$\limsup_{r\to0}\frac{\sup_{|x|\le r} \int_{\{|z-z_0|\le \varepsilon\}}|\rho(z)-\rho(x+z)|\,dz}{r}<\infty$$
for some $z_0\in\R^d$ and some $\varepsilon>0$, where $\rho(z)$ is a
Borel measurable function on $\R^d\setminus \{0\}$ such that
$\nu(dz)\ge \rho(z)\,dz$. \end{remark}

\medskip

The following result presents the exponential ergodicity for
Ornstein-Uhlenbeck processes, under weaker integral conditions for
the L\'{e}vy measure $\nu$ on the range $\{z\in\R^d: |z|\ge1\}$.

\begin{theorem}\label{erg2333}Let $X=(X^x_t)_{t\ge0}$ be a $d$-dimensional Ornstein-Uhlenbeck
process determined by \eqref{ou1}, where the real parts of all eigenvalues of $A$ are
negative. If the L\'{e}vy measure $\nu$ satisfies \begin{equation}\label{coup1erg2333}
       \liminf\limits_{|\xi|\rightarrow\infty}\frac{\int_{\{|z|\le 1/|\xi|\}}\langle z,\xi\rangle^2\,\nu(dz)}{\log (1+|\xi|)} >0.
\end{equation} and $\int_{\{|z|\ge1\}}|z|^\alpha\,\nu(dz)<\infty$ for some constant $0<\alpha\le1$,
then there are a unique invariant measure $\mu$ and two constants
$\kappa,C>0$ such that for any $x\in\R^d$ and $t>0$,
$$\|P_t(x, \cdot)-\mu\|_{\var}\le C (1+|x|^\alpha)\exp\left(-\kappa\,t\right).$$  \end{theorem}

It is clear that Theorem \ref{erg2333} can apply to study the
exponential ergodicity for Ornstein-Uhlenbeck processes driven by
$\alpha$-stable processes with $\alpha\in (0,2)$.

\bigskip

The remaining part of this paper is organized as follows. Section
\ref{section3} is devoted to the coupling property of
Ornstein-Uhlenbeck processes, which is key to our main results. In
Section \ref{section5}, we will present the proofs of Theorems
\ref{erg1} and \ref{erg2333}. Here, a general conclusion for the
exponential ergodicity of Ornstein-Uhlenbeck processes is given (see
Theorem \ref{thehreheh} below), which improves the second assertion
in Theorem \ref{erg1}.

\section{Coupling property}\label{section3}

In this section, we are mainly concerned with the coupling property
for the Ornstein-Uhlenbeck process $X=(X^x_t)_{t\ge0}$ given by
\eqref{ou1ou}. Recall that the process $X$ has successful
couplings (or has the coupling property) if and only if for any
$x,y\in\R^d$,
\begin{equation*}\label{prex1}
    \lim_{t\rightarrow\infty}\|P_t(x,\cdot)-P_t(y,\cdot)\|_{\var}=0,
\end{equation*}
where $P_t(x,dz)$ is the transition kernel of the process $X$ and
$\|\cdot\|_{\var}$ stands for the total variation norm. The coupling
property has been intensively studied for L\'{e}vy processes on
$\R^d$ and Ornstein-Uhlenbeck processes driven by L\'{e}vy processes
on $\R^d$, see \cite{BSW, SSW, SW1, SW2,W1}. Recently, by using the
lower bound conditions for the L\'{e}vy measure with respect to a
nice reference probability measure, we have successfully obtained
the coupling property for linear stochastic differential equations
driven by non-cylindrical L\'{e}vy processes on Banach spaces, see
\cite[Theorem 1.2]{WW}.

\medskip

Let $\nu$ be the L\'{e}vy measure corresponding to the L\'{e}vy
process $(Z_t)_{t\ge0}$, see \eqref{ou2}. For every $\varepsilon>0$,
define a finite measure ${\nu}_\varepsilon$ on $\R^d$ as that in
\eqref{levycut}. For a $d\times d$ matrix $A$, we say that an
eigenvalue $\lambda$ of $A$ is \emph{semisimple} if the dimension of
the corresponding eigenspace is equal to the algebraic multiplicity
of $\lambda$ as a root of characteristic polynomial of $A$. Note
that for symmetric matrices all eigenvalues of them are real and
semisimple.

\bigskip

The following result generalizes \cite[Theorem 1.1]{SW2}, and it
presents the exponential rate for the coupling property of
Ornstein-Uhlenbeck processes.

\begin{theorem}\label{coupgeneral} Let $X=(X^x_t)_{t\ge0}$ be the Ornstein-Uhlenbeck process given by
\eqref{ou1ou}, where the real parts of all eigenvalues of $A$ are
non-positive and all purely imaginary eigenvalues of $A$ are
semisimple. If there exist two constants $\varepsilon, \rho>0$
such that
\begin{equation}\label{eq21}\inf_{ |x|\le \rho}\nu
_\varepsilon\wedge
(\delta_{x}*\nu_\varepsilon)(\R^d)>0,\end{equation} then there
exists a constant $C_1>0 $ such that for all $x,y\in\R^d$ and
$t>0$,
$$
    \|P_t(x,\cdot)-P_t(y,\cdot)\|_{\var}\le
  \frac{C_1(1+|x-y|)}{\sqrt{t}}.$$

  Furthermore, suppose that the real parts of all eigenvalues of $A$ are
negative. If \eqref{eq21} is strengthened by
  \begin{equation}\label{eq2122ss}\limsup_{\rho\to0} \left[\frac{\sup\limits_{ |x|\le \rho}\|\nu
_\varepsilon-
(\delta_{x}*\nu_\varepsilon)\|_{_{\var}}}{\rho}\right]<\infty,\end{equation}
then there exist two constants $\kappa, C_2>0 $ such that for all
$x,y\in\R^d$ and $t>0$,
$$    \|P_t(x,\cdot)-P_t(y,\cdot)\|_{\var}\le
{C_2(1+|x-y|)}\exp\left(-\kappa\,t\right).$$
  \end{theorem}

\begin{remark}\label{remark3344} A condition \eqref{eq21} has been used to study the coupling
property for L\'{e}vy processes on $\R^d$ (see \cite[Theorem
1.1]{SW1}), and  Ornstein-Uhlenbeck
driven by L\'{e}vy processes on $\R^d$ (see \cite[Theorem
1.1]{SW2}).
 \end{remark}

\begin{proof}[Proof of Theorem \ref{coupgeneral}] The first required assertion has
been proven in \cite[Theorem 1.1]{SW2}, and so it suffices to prove
the second one. For simplicity, denote by $T_t=e^{tA}$ for $t\ge0.$
Since the real parts of all eigenvalues of $A$ are negative,
$$\|T_t\|_{\R^d\to\R^d}:=\sup_{x\in\R^d, |x|=1}|T_tx|\le ce^{-\lambda
t}$$ for all $t>0$ and some constants $c, \lambda>0$, e.g.\ see
\cite[(2.8)]{SA1}. For any $\varepsilon>0$, let
$(Z_t^\varepsilon)_{t\ge0}$ be a compound Poisson process on $\R^d$
with L\'{e}vy measure $\nu_\varepsilon$, which is well defined since
$\nu_\varepsilon$ is a finite measure on $\R^d$. Then,
$(Z_t^\varepsilon)_{t\ge0}$ and $(Z_t-Z_t^\varepsilon)_{t\ge0}$ are two
independent L\'{e}vy processes on $\R^d$. It follows, in particular,
that the random variables
$$
    X_t^{\varepsilon,x}:=T_tx+\int_0^t T_{t-s}\,dZ_s^\varepsilon
$$
and
$$
   \overline{ X_t^{\varepsilon} }:=X_t^x-X_t^{\varepsilon,x}=\int_0^t  T_{t-s}\,d(Z_s-Z_s^\varepsilon)
$$
are well defined on $\R^d$ and are independent for any
$\varepsilon>0$ and $t\ge0$.

Denote by
$$
    X_t^{\varepsilon,0} := X_t^{\varepsilon,x}-T_tx = \int_0^tT_{t-s}\,dZ_s^\varepsilon.
$$
We will rewrite  $X_t^{\varepsilon,0}$ as follows. Construct a
sequence $(\tau_i)_{i\ge1}$ of i.i.d.\ random variables which are
exponentially distributed with intensity
$C_\varepsilon=\nu_\varepsilon(\R^d)$, and introduce a further
sequence $(U_i)_{i\ge1}$ of i.i.d.\ random variables on $\R^d$ with
law $\bar{\nu}_\varepsilon=\nu_\varepsilon/C_\varepsilon$. We will
assume that the random variables $(U_i)_{i\ge1}$ are independent of
the sequence $(\tau_i)_{i\ge1}$. Then, according to \cite[Theorem
4.3]{Sa},
$$Z_t^\varepsilon=\sum_{i=1}^{N_t}U_i$$ for every $t\ge0$, where
$(N_t)_{t\ge0}$ is a Poisson process of intensity $C_\varepsilon$,
i.e.\ $$N_t=\sup\Big\{k\ge1: \sum_{i=1}^k\tau_i\le t\Big\},$$ and
here we set $\sum_{i\in \varnothing}=0$ by convention. Therefore, it
is not difficult to check that
\begin{equation*}\label{proofs0}
  X_t^{\varepsilon,0}= 0\cdot\I_{\{\tau_1>t\}}
  +\sum_{k=1}^\infty\I_{\{\tau_1+\cdots+\tau_k\le t<\tau_1+\cdots+\tau_{k+1}\}}
   \Big(T_{t-\tau_1} U_1+\cdots+T_{t-(\tau_1+\cdots+\tau_k)}
   U_k\Big).
\end{equation*}
Since $$\tau_1+\cdots+\tau_{N_t}\le
t<\tau_1+\cdots+\tau_{N_t+1},\quad t\ge0,$$ it holds that on the set $\{N_t\ge1\}$,
$$X_t^{\varepsilon,0}=\sum_{k=1}^{N_t}T_{t-\sum_{i=1}^k\tau_i}U_k.$$

Next, we will make use of the decomposition:
\begin{equation}\label{coupgeneral11}P_tg(x)=\Ee(g(X_t^x)\I_{\{N_t=0\}})+P_t^1g(x),\quad g\in
B_b(\R^d), t\ge0, x\in\R^d,\end{equation} where
\begin{equation}\label{modisemgroup11} P_t^1g(x)=\Ee(g(X_t^x)\I_{\{N_t\ge1\}}).\end{equation}
According to all the statements above, we know that for any $g\in
B_b(\R^d)$ and $x\in \R^d$,
$$\aligned P_t^1g(x)=&\Ee\left( \I_{\{N_t\ge1\}}g\Big(T_tx+  \overline{ X_t^{\varepsilon}
}+X_t^{\varepsilon,0}\Big)
\right)\\
=&\Ee\left( \I_{\{N_t\ge1\}}g\Big(T_tx+  \overline{
X_t^{\varepsilon}
}+\sum_{k=1}^{N_t}T_{t-\sum_{i=1}^k\tau_i}U_k\Big)
\right)\\
=&\Ee\left(\I_{\{N_t\ge1\}} g\Big(T_tx+  \overline{
X_t^{\varepsilon}
}+\sum_{k=1}^{N_t-1}T_{t-\sum_{i=1}^k\tau_i}U_k+T_{t-\sum_{i=1}^{N_t}\tau_i}U_{N_t}\Big)
\right)\\
=&\frac{1}{C_\varepsilon}\Ee\left(\I_{\{N_t\ge1\}}\int_{\R^d}
g\Big(T_tx+ \overline{ X_t^{\varepsilon}
}+\sum_{k=1}^{N_t-1}T_{t-\sum_{i=1}^k\tau_i}U_k+T_{t-\sum_{i=1}^{N_t}\tau_i}z\Big)\,\nu_\varepsilon(dz)\right).\endaligned$$
Therefore, for any $x,y\in \R^d,$
$$\aligned &\big|P_t^1g(x)-P_t^1g(y)\big|\\
&=\frac{1}{C_\varepsilon}\Bigg[\Ee\left(\I_{\{N_t\ge1\}}\int_{\R^d}
g\Big(T_tx+ \overline{ X_t^{\varepsilon}
}+\sum_{k=1}^{N_t-1}T_{t-\sum_{i=1}^k\tau_i}U_k+T_{t-\sum_{i=1}^{N_t}\tau_i}z\Big)\,\nu_\varepsilon(dz)\right)\\
&\qquad\quad-\Ee\left(\I_{\{N_t\ge1\}}\int_{\R^d}  g\Big(T_ty+
\overline{ X_t^{\varepsilon}
}+\sum_{k=1}^{N_t-1}T_{t-\sum_{i=1}^k\tau_i}U_k+T_{t-\sum_{i=1}^{N_t}\tau_i}z\Big)\,\nu_\varepsilon(dz)\right)\Bigg]\\
&=\frac{1}{C_\varepsilon}\Bigg\{\Ee\Bigg[\I_{\{N_t\ge1\}}\\
&\qquad \quad\times\!\!\bigg(\int_{\R^d}\!g\Big(T_ty\!+\!\!
\overline{ X_t^{\varepsilon}
}\!+\!\sum_{k=1}^{N_t-1}T_{t-\sum_{i=1}^k\tau_i}U_k\!+\!T_{t-\sum_{i=1}^{N_t}\tau_i}
\!\!\Big(z+T_{\sum_{i=1}^{N_t}\tau_i}(x-y)\Big)\Big)\,\nu_\varepsilon(dz)\\
&\qquad\qquad\quad -\int_{\R^d}g\Big(T_ty+ \overline{
X_t^{\varepsilon}
}+\sum_{k=1}^{N_t-1}T_{t-\sum_{i=1}^k\tau_i}U_k+T_{t-\sum_{i=1}^{N_t}\tau_i}z\Big)\,\nu_\varepsilon(dz)\bigg)\Bigg]\Bigg\}\\
&=\frac{1}{C_\varepsilon}\Bigg\{\Ee\Bigg[\I_{\{N_t\ge1\}}\\
&\qquad\quad\times\!\!\bigg(\int_{\R^d}g\Big(T_ty\!+\!\!
\overline{ X_t^{\varepsilon}
}\!+\!\!\sum_{k=1}^{N_t-1}T_{t-\sum_{i=1}^k\tau_i}U_k\!+\!\!T_{t-\sum_{i=1}^{N_t}\tau_i}z\Big)\,\nu_\varepsilon\big(dz-T_{\sum_{i=1}^{N_t}\tau_i}(x-y)\big)\\
&\qquad\qquad\quad -\int_{\R^d}g\Big(T_ty+ \overline{
X_t^{\varepsilon}
}+\sum_{k=1}^{N_t-1}T_{t-\sum_{i=1}^k\tau_i}U_k+T_{t-\sum_{i=1}^{N_t}\tau_i}z\Big)\,\nu_\varepsilon(dz)\bigg)\Bigg]\Bigg\}\\
&\le \frac{(1-e^{-C_\varepsilon
t})}{C_\varepsilon}\|g\|_{\infty}\sup_{z\in \R^d, |z|\le
c|x-y|}\|\nu_\varepsilon-\delta_z*\nu_\varepsilon\|_{\var}\\
&\le \|g\|_{\infty}\frac{c\Gamma}{C_\varepsilon}|x-y|.\endaligned$$
Here, in the first inequality we have used the facts that
$\Pp(N_t\ge1)=1-e^{-C_\varepsilon t}$ for $t\ge0$, and
$\|T_t\|_{\R^d\to\R^d}\le c$ for all $t\ge 0$; and in the last
inequality we set
$$\Gamma:=\sup_{\rho>0} \left[\frac{\inf\limits_{|x|\le \rho}\|\nu
_\varepsilon-
(\delta_{x}*\nu_\varepsilon)\|_{_{\var}}}{\rho}\right],$$ which is
finite due to \eqref{eq2122ss} and the fact that
 $$\sup_{|x|\le \rho}\|\nu
_\varepsilon- (\delta_{x}*\nu_\varepsilon)\|_{_{\var}}\le
2C_\varepsilon,\quad \rho>0, \varepsilon>0.$$ On the other hand,
we have
$$|\Ee(g(X_t^x)\I_{\{N_t=0\}})|\le\|g\|_{\infty}
e^{-C_\varepsilon t},\quad t\ge0, g\in B_b(\R^d).$$ Combining all
the estimates with \eqref{coupgeneral11}, we get that for any
$x,y\in \R^d,$
$$\big|P_tg(x)-P_tg(y)\big|\le 2\|g\|_{\infty}e^{-C_\varepsilon
t}+ \frac{ c\Gamma }{C_\varepsilon}\|g\|_{\infty}|x-y|.$$

Having all the conclusions above at hand, we can follow the proof
of \cite[Theorem 1.3]{WW} to get the desired assertion. Since
$\|T_t\|_{\R^d\to \R^d}\le c e^{-\lambda t}$ for all $t\ge0$ and
some constants $c,\lambda>0$, it follows from \eqref{ou1ou} that
$$|X_t^x-X_t^y|\le ce^{-\lambda t}
|x-y|,\quad x,y\in \R^d, t\ge0.$$ Therefore, for any $0<s<t$ and
$x,y\in\R^d,$
$$\aligned\big|P_tg(x)-P_tg(y)\big|&=\Ee\big|P_sg(X_{t-s}^x)-P_sg(X_{t-s}^y)\big|\\
&\le 2\|g\|_{\infty}e^{-C_\varepsilon s}+
\frac{c\Gamma}{C_\varepsilon}\|g\|_{\infty}
|X_{t-s}^x-X_{t-s}^y|\\
&\le c_1\|g\|_{\infty} (1+|x-y|) \big(e^{-C_\varepsilon s}\vee
e^{-\lambda (t-s)}\big)
\endaligned$$ holds for some constant $c_1>0$. Setting $s=\frac{\lambda t}{C_\varepsilon +\lambda}$, we get
the required assertion.
\end{proof}

According to the proof above, under condition \eqref{eq2122ss} one
can get the following gradient estimates for a modified version
$(P_t^1)_{t\ge0}$ of $(P_t)_{t\ge0}$ (see \eqref{modisemgroup11}):
$$\sup_{t\ge0,\|g\|_\infty=1,x\in\R^d}|\nabla P_t^1g(x)|:=\sup_{t\ge0,\|g\|_\infty=1,x\in\R^d}\limsup_{x\to
y}\frac{|P_t^1g(y)-P_t^1g(x)|}{|y-x|}<\infty.$$ Such estimates have
been considered in \cite[Theorem 3.1]{W2} and \cite[Proposition
4.1]{WW} by using the formula for random shifts of the compound
Poisson measures, when the L\'{e}vy measure is required to have
absolutely continuous lower bounds with respect to some nice
reference measures, e.g.\ the Lebesgue measure on $\R^d$ or the
Gaussian measure on the Wiener space. Here, our condition
\eqref{eq2122ss} is more general and the proof is more direct.

\medskip

A close inspection of the proof of Theorem \ref{coupgeneral} shows

\begin{corollary}\label{th2gradient}Let $X=(X^x_t)_{t\ge0}$ be the Ornstein-Uhlenbeck process given by
\eqref{ou1ou}, where the real parts of all eigenvalues of $A$ are
non-positive and all purely imaginary eigenvalues of $A$ are
semisimple. If there exists a finite measure $\mu$ on $\R^d$ such that $\nu\ge \mu$ and
 \begin{equation*}\label{eq2122}\limsup_{\rho\to0} \left[\frac{\sup\limits_{ |x|\le \rho}\|\mu-
(\delta_{x}*\mu)\|_{_{\var}}}{\rho}\right]<\infty,\end{equation*}
then there exist two constants $\kappa, C>0 $ such that for all
$x,y\in\R^d$ and $t>0$,
$$    \|P_t(x,\cdot)-P_t(y,\cdot)\|_{\var}\le
{C(1+|x-y|)}\exp\left(-\kappa\,t\right).$$
\end{corollary}

\bigskip

The following estimate $\|P_t(x,\cdot)-P_t(y,\cdot)\|_{\var}$ for large values of $t$ is based on the characteristic exponent $\Phi(\xi)$ of the L\'evy process $(Z_t)_{t\ge0}$.

\begin{theorem}\label{coup111}
   Let $X=(X^x_t)_{t\ge0}$ be a $d$-dimensional Ornstein-Uhlenbeck
process determined by \eqref{ou1}, where the real parts of all
eigenvalues of $A$ are negative. Assume that the L\'{e}vy measure
$\nu$ of $Z$ satisfies $\int_{\{|z|\ge1\}}\log
(1+|z|)\,\nu(dz)<\infty$, and the associated symbol $\Phi$ fulfills
\begin{equation*}\label{coup1erg2}
       \liminf\limits_{|\xi|\rightarrow\infty}\frac{\Re\Phi\big(\xi\big)}{\log (1+|\xi|)} >0.
\end{equation*}
   Then there exist $t_1,C>0$ such that for any $x,y\in\R^d$ and $t\ge t_1$,
\begin{equation*}\label{coup3}
    \|P_t(x,\cdot)-P_t(y,\cdot)\|_{\var}
    \le C|e^{tA}(x-y)|\,\varphi^{-1}_t(1),
\end{equation*} where for $t, \rho>0$,
$$
    \varphi_t(\rho)
    :=\sup_{|\xi|\le \rho}\int_0^t\Re\Phi\big(e^{sA^\top}\xi\big)\,ds,
$$
and $A^\top$ denotes the transpose of the matrix $A$.
\end{theorem}

\begin{proof} We first assume that the L\'{e}vy process $(Z_t)_{t\ge 0}$ is a pure jump process, i.e.\ $Q=0$ and $b=0$ in \eqref{ou2}. According to \cite[Theorem 1.7]{SW2}, it suffices to verify that  \begin{equation*}\label{coup4}\xi\mapsto\int_0^\infty\Re\Phi\big(
e^{sA^\top}\xi\big)\,ds \quad\textrm{ is locally
bounded,}\end{equation*} and there exists some $t_0 > 0$ such that
    \begin{equation*}\label{coup1}
       \liminf\limits_{|\xi|\rightarrow\infty}\frac{\int_0^{t_0}\Re\Phi\big( e^{sA^\top}\xi\big)\,ds}{\log (1+|\xi|)} >2d+2.
\end{equation*}

First, since the driving L\'evy process $(Z_t)_{t\ge0}$ has no
Gaussian part, according to \cite[Theorem 4.1]{SA1} (or
\cite[Proposition 2.2]{Mu}) and the assumptions, the
process $(X_t)_{t\ge0}$ possesses an invariant measure $\mu$, which
is an infinite divisible distribution with the characteristic
exponent $\xi\mapsto\int_0^\infty\Phi\big( e^{sA^\top}\xi\big)\,ds$. In
particular, the function $$\xi \mapsto \int_0^\infty\Re\Phi\big(
e^{sA^\top}\xi\big)\,ds$$ is well defined and locally bounded.

On the other hand, set
$$c_0:=\liminf\limits_{|\xi|\rightarrow\infty}\frac{\Re\Phi\big(\xi\big)}{\log (1+|\xi|)} >0.$$ Choosing $t_0> \frac{2d+2}{c_0},$ we have
$$\aligned \varliminf\limits_{|\xi|\rightarrow\infty}\frac{\int_0^{t_0}\Re\Phi\big( e^{sA^\top}\xi\big)\,ds}{\log (1+|\xi|)}&\ge\varliminf\limits_{|\xi|\rightarrow\infty}\int_0^{t_0}\!\!\frac{\Re\Phi\big( e^{sA^\top}\xi\big)}{\log (1+|e^{sA^\top}\xi|)}\,ds\,\frac{\inf_{0<s<t_0}\log (1+|e^{sA^\top}\xi|)}{\log(1+|\xi|)}\\
&\ge \int_0^{t_0}\!\!\varliminf\limits_{|\xi|\rightarrow\infty}\frac{\Re\Phi\big( e^{sA^\top}\xi\big)}{\log (1+|e^{sA^\top}\xi|)}\,ds\\
&>2d+2,
\endaligned$$ where the second inequality follows from the Fatou lemma and the fact that
$$\lim_{|\xi|\to \infty}\frac{\inf_{0<s<t_0}\log (1+|e^{sA^\top}\xi|)}{\log(1+|\xi|)}=1.$$
This proves the required assertion.

Next, we consider the general case. Let $(Y_t)_{t\ge0}$ and
$(Z_t)_{t\ge0}$ be two independent L\'evy processes, whose symbols
are  $$\Phi_Y(\xi)
     =\int_{z\neq 0} \Bigl(1-e^{i\scalp{\xi}{z}}+i\scalp{\xi}{z}\I_{B(0,1)}(z)\Bigr)\nu(dz), $$ and $$\Phi_Z(\xi)=\Phi(\xi)-\Phi_Y(\xi),$$ respectively.
 Denote by $Q_{t}$ and $Q_{t}(x,\cdot)$ the semigroup and the transition function of the $d$-dimensional Ornstein-Uhlenbeck
process driven by $(Y_t)_{t\ge0}$. Similarly,  $R_{t}$ and
$R_{t}(x,\cdot)$ stand for the semigroup and the transition function
of the $d$-dimensional Ornstein-Uhlenbeck process driven by
$(Z_t)_{t\ge0}$. Note that $Q_{t}(x,\cdot)$ is the transition kernel
of an Ornstein-Uhlenbeck process driven by pure jump L\'{e}vy
process. Then,
\begin{equation*}\label{pcoup0}\begin{aligned}
    \|P_t(x,\cdot)-P_t(y,\cdot)\|_{\var}
    &= \sup_{\|f\|_\infty\le 1}\big|P_tf(x)-P_tf(y)\big|\\
    &= \sup_{\|f\|_\infty\le 1} \big|Q_t R_t f(x)-Q_t R_tf(y)\big|\\
    &\le \sup_{\|g\|_\infty\le 1} \big|Q_tg(x)-Q_tg(y)\big|\\
    &= \|Q_t(x,\cdot)-Q_t(y,\cdot)\|_{\var}.
\end{aligned}\end{equation*}
This, along with the conclusion above for $Q(x,dz)$, completes the
proof.
\end{proof}
\section{Proofs}\label{section5} In the section, we will apply the
results in Section \ref{section3} to study the ergodicity and the
exponential ergodicity for Ornstein-Uhlenbeck processes. It is well known that the coupling property along with
the existence of an stationary measure can yield the ergodicity
for the process, which gives us the motivation of the proof of
Theorem \ref{erg1}.

\begin{proof}[Proof of Theorem \ref{erg1}] As mentioned in Theorem 1.0, according to
\cite[Theorem 4.1]{SA1} or \cite[Proposition 2.2]{Mu}, the process
$X$ has an invariant measure $\mu$. In particular, $\mu P_t(\cdot,dz)=\mu(dz)$ for
any $t>0$, where $P_t(x,dz)$ is the transition kernel of
the process $X$. On the other hand, by Theorem \ref{coupgeneral},
\eqref{erg11} and the condition that the real parts of all
eigenvalues of $A$ are negative imply that there exists $C_1>0$ such
that for any $t>0$ and $x,y\in\R^d$,
$$\|P_t(x,\cdot)-P_t(y,\cdot)\|_{\var}\le \frac{C_1(1+|x-y|)}{\sqrt{t}}.$$ That is, when $t\to \infty$,
$\|P_t(x,\cdot)-P_t(y,\cdot)\|_{\var}$ converges to zero uniformly
for all $x,y\in \R^d$ with bounded $|x-y|$. Note that for any
$x\in \R^d$ and $t>0$,
$$\|P_t(x, \cdot)-\mu\|_{\var}\le \int\|P_t(x, \cdot)-P_t(y,\cdot)\|_{\var}\,\mu(dy).$$ This along with the statement above
gives us that for any $x\in \R^d$,
$$\lim_{t\to\infty}\|P_t(x, \cdot)-\mu\|_{\var}=0.$$ We mention here that the proof above also yields the
uniqueness of invariant measure. Indeed, let $\mu_1$ and $\mu_2$
be invariant measures for the process $X$. Then,
$$\|\mu_1-\mu_2\|_{\var}\le \int\|P_t(x,
\cdot)-P_t(y,\cdot)\|_{\var}\,\mu_1(dx)\,\mu_2(dy).$$ Combining with the proof above and letting $t\to\infty$ show that $\mu_1=\mu_2$. This proves the first required assertion.

For the second assertion, by \eqref{erg12} and Theorem
\ref{coupgeneral}, we know that there exist $\theta,C_2>0$ such that for
any $t>0$ and $x,y\in\R^d$,
$$\|P_t(x,\cdot)-P_t(y,\cdot)\|_{\var}\le {C_2(1+|x-y|)}e^{-\theta
t}.$$ We will claim that, under the assumption
$\int_{\{|z|\ge1\}}|z|\,\nu(dz)<\infty,$
\begin{equation}\label{prooffff}\int
|x|\,\mu(dx)<\infty.\end{equation} If this holds, then following
the argument above, we have
$$\aligned\|P_t(x, \cdot)-\mu\|_{\var}&\le \int\|P_t(x, \cdot)-P_t(y,\cdot)\|_{\var}\,\mu(dy)\\
&\le {C_2(1+|x|)}e^{-\theta t}\int|y|\,\mu(dy)\\
&\le C_3(1+|x|)e^{-\theta t}.\endaligned$$ The required assertion
follows.

Next, we turn to prove \eqref{prooffff}. For $t>0$, set $T_t=e^{tA}$ and
$Y_t=\int_0^t T_{t-s}\, dZ_s$. For simplicity, we assume that $Z$
is a L\'{e}vy process on $\R^d$ without the Gaussian part. Thus,
according to the L\'{e}vy-It\^{o} decomposition (see \cite[Chapter 4]{Sa}), there are $b\in \R^d$ and a Poisson random measure $N$
on $[0,\infty)\times \R^d\setminus\{0\}$ with intensity measure
$ds\otimes \nu(dz)$ (where $ds$ is the Lebesgue measure on
$[0,\infty)$) such that
$$dZ_s=b\,ds+\int_{\{|z|\le 1\}}z \,\widetilde{N}(ds, dz)+\int_{\{|z|> 1\}}z \,N(ds, dz),$$ where $\widetilde{N}(ds,dz)$ is the compensated Poisson measure on $[0,\infty)\times \R^d\setminus\{0\}$, i.e.\ $$\widetilde{N}(ds,dz)=N(ds,dz)-ds\,\nu(dz).$$ Hence, the integral $Y_t$ is defined by
$$Y_t=\int_0^t  T_{t-s} b\,ds+\int_0^t\int_{\{|z|\le 1\}}T_{t-s}z \,\widetilde{N}(ds, dz)+ \int_0^t\int_{\{|z|> 1\}}T_{t-s}z \,N(ds, dz).$$ Since the real parts of all eigenvalues of $A$ are
negative, $\|T_t\|_{\R^d\to\R^d}\le ce^{-\lambda t}$ for all $t>0$ and some constants $c, \lambda>0$, e.g.\ see \cite[(2.8)]{SA1}. Thus, for any $t>0$,
$$\bigg|\int_0^t  T_{t-s} b\,ds\bigg|\le \int_0^t \|T_{t-s}\|_{\R^d \to \R^d} |b|\,ds\le c|b|\int_0^t e^{-\lambda (t-s)}\,ds\le \frac{c}{\lambda} |b|,$$ and,
by using the Cauchy-Schwarz inequality and the fact that
$\widetilde{N}(ds,dz)$ is a square integrable martingale measure,
cf.\ see \cite[Chapter 4.2]{App1},
$$\aligned \Ee\bigg|\int_0^t\int_{\{|z|\le 1\}}T_{t-s}z  \,\widetilde{N}(ds, dz)\bigg|&\le  \bigg(\Ee\bigg|\int_0^t\int_{\{|z|\le 1\}}T_{t-s}z  \,\widetilde{N}(ds, dz)\bigg|^2\bigg)^{1/2}\\
&=\bigg(\int_0^t\int_{\{|z|\le 1\}} |T_{t-s}z|^2\,\nu(dz)\,ds\bigg)^{1/2}\\
&\le\bigg(\int_0^t\|T_{t-s}\|_{\R^d\to \R^d}^2\,ds\int_{\{|z|\le 1\}} |z|^2\,\nu(dz)\bigg)^{1/2}\\
&\le\frac{c}{\lambda}\sqrt{\int_{\{|z|\le 1\}} |z|^2\,\nu(dz)}
.\endaligned$$ On the other hand, noting that the integral
$\int_0^t\int_{\{|z|> 1\}}T_{t-s}z \,N(ds, dz)$ is defined as Remain
integral and $\int_0^t\int_{\{|z|> 1\}}z \,N(ds, dz)$ is a compound
Poisson process with intensity $\nu(\{z\in \R^d, |z|>1\}),$ it
follows from the argument in \cite[Chapter 4.3.5]{App1} that
$\int_0^t\int_{\{|z|> 1\}}T_{t-s}z \,N(ds, dz)$ is an infinitely
divisible random variable associated with a L\'{e}vy measure
$$\nu_t(D):=\int_0^t \int T_{t-s}^{-1}({D\cap \{z\in \R^d:
|z|>1\}})\,\nu(dz)\quad\textrm{ for } D\in \R^d\setminus\{0\}.$$ Thus,
$$\aligned\int  \,\nu_t(dz)=&\int_0^t \int_{\{| T_{t-s}z|>1\}}\,\nu(dz)\,ds\\
\le& \int_0^t \int_{\{|z|>c^{-1}\}}\,\nu(dz)\,ds\\
\le& t\int_{\{|z|>c^{-1}\}}\,\nu(dz).\endaligned$$ That is, for any
$t>0$, the L\'{e}vy measure $\nu_t$ is a finite measure.

Therefore,
$\int_0^t\int_{\{|z|> 1\}}T_{t-s}z \,N(ds, dz)$ can be regarded as
random variable $E_{t,1}$ for some compound Poisson process
$(E_{t,s})_{s\ge0}$ with bounded L\'{e}vy measure $\nu_t$. According
to the explicit expression of the semigroup for compound Poisson
process, see the proof of \cite[Theorem 25.3]{Sa} or
\cite[(2.1)]{SW1}, we get that
$$\Ee \bigg|\int_0^t\int_{\{|z|> 1\}}T_{t-s}z \,N(ds, dz)\bigg|\le \sum^\infty_{n=0}\frac{1}{n!}\int |z| \,{\nu_t^*}^{n}(dz)\le \sum^\infty_{n=0}\frac{1}{n!}\Big(\int |z| \,\nu_t(dz)\Big)^n,$$ where ${\nu_t^*}^{n}$ is the $n$-fold convolution of $\nu_t$ and ${\nu_t^*}^0=\delta_0.$ By using the fact that for any $t>0$, $$\aligned\int |z| \,d\nu_t(dz)=&\int_0^t \int_{\{| T_{t-s}z|>1\}}|T_{t-s}z|\,\nu(dz)\,ds\\
\le& \int_0^t \|T_{t-s}\|_{\R^d\to \R^d}\,ds\int_{\{|z|>c^{-1}\}} |z|\,\nu(dz)\\
\le& \frac{c}{\lambda}\int_{\{|z|>c^{-1}\}}|z|\,\nu(dz),\endaligned$$ we arrive at
$$\Ee \bigg|\int_0^t\int_{\{|z|> 1\}}T_{t-s}z \,N(ds, dz)\bigg|\le\exp \bigg(\frac{c}{\lambda}\int_{\{|z|>c^{-1}\}}|z|\,\nu(dz)\bigg).$$
Combining with all the conclusions above, we get that $\Ee
\big|Y_t\big|$ is bounded uniformly for all $t>0$, i.e. $\sup_{t>0}\Ee\big|Y_t\big|\le C_0$ for some absolutely constant $C_0$.

Furthermore, by \eqref{ou1ou}, for any $m\ge1$ and $t>0$,
$$|X_t^x|\wedge m\le |T_tx|\wedge m+|Y_t|,$$ and so
$$\Ee \big( |X_t^x|\wedge m\big)\le \Ee \big( |T_tx|\wedge m\big)+\Ee |Y_t|\le \Big(ce^{-\lambda t} |x|\Big)\wedge m+C_0.$$ Integrating this inequality with $\mu(dx)$, we get that
$$\mu(|x|\wedge m)\le \mu\Big[\Big(ce^{-\lambda t} |x|\Big)\wedge m\Big]+C_0,\quad t>0, m\ge1.$$ Letting first $t\to \infty$ and then $m\to \infty$, we prove the required assertion \eqref{prooffff}. The proof is finished.
\end{proof}

We note that the argument of Theorem \ref{erg1} above yields that:

\medskip

\begin{corollary}{Let $X=(X^x_t)_{t\ge0}$ be a $d$-dimensional Ornstein-Uhlenbeck
process determined by \eqref{ou1}, where the real parts of all
eigenvalues of $A$ are negative. If the L\'{e}vy measure $\nu$ of
the L\'{e}vy process $Z$ satisfies \eqref{erg11} and
$\int_{\{|z|\ge1\}}|z|\,\nu(dz)<\infty,$ then the process $X$ is
ergodic in the sense of algebraic convergence, i.e.\ there exist a
unique invariant measure $\mu$ and a positive constant $C$ such that
for any $x\in\R^d$ and $t>0$,
$$\|P_t(x, \cdot)-\mu\|_{\var}\le \frac{C (1+|x|)}{\sqrt{t}}.$$}\end{corollary}

Furthermore, according to Corollary \ref{th2gradient} and the proof of Theorem \ref{erg1}, we have the following conclusion for the exponential ergodicity of Ornstein-Uhlenbeck
processes, which improves the second assertion in Theorem \ref{erg1}.

\begin{theorem}\label{thehreheh} Let $X=(X^x_t)_{t\ge0}$ be a $d$-dimensional Ornstein-Uhlenbeck
process determined by \eqref{ou1}, where the real parts of all eigenvalues of $A$ are
negative and the L\'{e}vy
measure $\nu$ of the L\'{e}vy process $Z$ satisfies $\int_{\{|z|\ge1\}}|z|\,\nu(dz)<\infty.$ If there exists a finite measure $\mu$ on $\R^d$ such that $\nu\ge \mu$ and
 \begin{equation*}\label{eq2122}\limsup_{\rho\to0} \left[\frac{\sup\limits_{ |x|\le \rho}\|\mu-
(\delta_{x}*\mu)\|_{_{\var}}}{\rho}\right]<\infty,\end{equation*}
then there are a unique invariant measure $\mu$ and two constants
$\kappa,C>0$ such that for any $x\in\R^d$ and $t>0$,
$$\|P_t(x, \cdot)-\mu\|_{\var}\le C (1+|x|)\exp\left(-\kappa\,t\right).$$  \end{theorem}

\bigskip

The proof of Theorem \ref{erg2333} is based on the following lemma.

\begin{lemma}\label{erg2}Let $X=(X^x_t)_{t\ge0}$ be a $d$-dimensional Ornstein-Uhlenbeck
process determined by \eqref{ou1}, where the real parts of all
eigenvalues of $A$ are negative, and the symbol of the L\'{e}vy
process $Z$ satisfies  \begin{equation}\label{coup1erg2}
       \liminf\limits_{|\xi|\rightarrow\infty}\frac{\Re\Phi\big(\xi\big)}{\log (1+|\xi|)} >0.
\end{equation} If there exists a constant $0<\alpha\le1$ such that the L\'{e}vy measure $\nu$ of
L\'{e}vy process $Z$ satisfies
$\int_{\{|z|\ge1\}}|z|^\alpha\,\nu(dz)<\infty$, then there are a
unique invariant measure $\mu$ and two constants $\kappa,C>0$ such
that for any $x\in\R^d$ and $t>0$,
$$\|P_t(x, \cdot)-\mu\|_{\var}\le C (1+|x|^\alpha)\exp\left(-\kappa\,t\right).$$ \end{lemma}

\begin{proof} For $t, \rho>0$, define
$$
    \varphi_t(\rho)
    :=\sup_{|\xi|\le \rho}\int_0^t\Re\Phi\big(e^{sA^\top}\xi\big)\,ds.
$$
According to Theorem \ref{coup111} and the Markov property, there exists a constant $t_1>0$ such that
for any $t\ge t_1$,
$s>0$ and $g\in B_b(\R^d)$,
$$\aligned |P_tg(x)-P_{t+s}g(x)|=&|\Ee (P_tg(x)-P_tg(X_s^x))|\\
\le &\Ee|P_tg(x)-P_tg(X_s^x)|\\
=&\Ee\bigg(\frac{|P_tg(x)-P_tg(X_s^x)|^\alpha}{|X_s^x-x|^\alpha} |P_tg(x)-P_tg(X_s^x)|^{1-\alpha}|X_s^x-x|^\alpha\bigg)\\
\le &C_t^\alpha \|g\|_{\infty}^\alpha \times  (2\|g\|_{\infty})^{1-\alpha}\times\Ee\big( |X_s^x-x|^\alpha\big)\\
=&2^{1-\alpha}C_t^\alpha\|g\|_{\infty}
\Ee\big(|T_sx-x+Y_s|^\alpha\big).
\endaligned$$ where $T_sx=e^{sA}x$, $Y_s=\int_0^s T_{s-u}\,dZ_u$ and
$$C_t=\|T_t\|_{\R^d\to\R^d}\,\varphi^{-1}_t(1).$$

Since the real parts of all eigenvalues of $A$ are
negative, $\|T_t\|_{\R^d\to\R^d}\le ce^{-\lambda t}$ for all $t>0$ and some constants $c, \lambda>0$, e.g.\ see \cite[(2.8)]{SA1}. Therefore, for any $\alpha\in(0,1]$, $$
\Ee\big(|T_sx-x+Y_s|^\alpha\big)\le\Ee\big((|T_sx-x|+|Y_s|)^\alpha\big)\le
|T_sx-x|^\alpha+\Ee|Y_s|^\alpha$$ and $$|T_sx-x|^\alpha\le
|T_sx|^\alpha+|x|^\alpha\le
(1+c^\alpha)|x|^\alpha,$$ where we have used the fact that
$$(a+b)^\alpha\le a^\alpha+b^\alpha,\quad a,b\ge0.$$ On the other hand, under the assumption that $\int_{\{|z|\ge 1\}}|z|^\alpha\nu(dz)<\infty$, one can follow the proof of \eqref{prooffff} to verify that $\Ee|Y_s|^\alpha$ is uniformly bounded for all $s>0$, i.e.\ $\sup_{s>0}\Ee|Y_s|^\alpha<\infty.$ Therefore, there exists a constant $C_0>0$ such that   $$\sup_{s>0}\Ee\big(|T_sx-x+Y_s|^\alpha\big)\le C_0(1+|x|^\alpha).$$ Combining with all the conclusions above, we get that
$$|P_tg(x)-P_{t+s}g(x)|\le 2^{1-\alpha}C_0(1+|x|^\alpha)C_t^\alpha\|g\|_{\infty}.$$ That is,
$$\|P_t(x,\cdot)-P_{t+s}(x,\cdot)\|_{\var}\le  2^{1-\alpha}C_0(1+|x|^\alpha)C_t^\alpha.$$ Letting $s\to \infty$ and noting that $\mu$ is the invariant measure of the process $X$, $$\|P_t(x,\cdot)-\mu\|_{\var}\le  2^{1-\alpha}C_0(1+|x|^\alpha)C_t^\alpha.$$

As mentioned above, $\|T_t\|_{\R^d\to\R^d}\le ce^{-\lambda t}$ for all $t>0$. On the other hand, since for $t\ge t_1$,
  $\varphi_t(\rho)\ge \varphi_{t_1}(\rho)$ and $\lim\limits_{\rho\to \infty}\varphi_{t_1}(\rho)=\infty$, it holds that $\varphi^{-1}_t(1)\le \varphi^{-1}_{t_1}(1)<\infty$ for any $t\ge t_1$. Therefore, there exists $C>0$ such that for any $t\ge t_1$, $C_t^\alpha\le Ce^{-\lambda t}$, which along with the conclusion above yields the required assertion.
\end{proof}

\begin{proof}[Proof of Theorem \ref{erg2333}] According to Lemma \ref{erg2}, it is sufficient to verify that
\eqref{coup1erg2333} implies that \eqref{coup1erg2}. For any
$\xi\in\R^d$ with $|\xi|$ large enough,
\begin{align*}\Re\Phi(\xi)
     =&\int_{z\neq 0} \Bigl(1-\cos{ \langle{\xi},{z}\rangle}\Bigr)\,\nu(dz)\\
     \ge&\int_{0<|z|\le /1|\xi|} \Bigl(1-\cos{ \langle{\xi},{z}\rangle}\Bigr)\,\nu(dz)\\
     \ge& \frac{\cos 1}{2}\int_{0<|z|\le /1|\xi|} \langle{\xi},{z}\rangle^2\,\nu(dz),\end{align*} where in the second inequality we have used the inequality that
      $$1-\cos r\ge \frac{\cos 1}{2} r^2, \quad |r|\le1.$$ This follows the desired assertion. \end{proof}

\acks
We would like to thank an anonymous referee whose careful reading helped to improve the presentation of this paper.
Financial support through National Natural Science Foundation of China (No.\ 11126350) and the Programme of Excellent Young Talents in Universities
of Fujian (No.\ JA10058 and JA11051)
is gratefully
acknowledged.

%
%
%
%

\end{document}